\documentclass[12pt,reqno]{amsart}
\usepackage{amssymb,amsmath,amsthm,newlfont,enumerate}
\usepackage{pxfonts}
\usepackage[OT2,T1]{fontenc}
\usepackage{pifont}

\usepackage{a4wide}

\usepackage{color}

\newtheorem{thm}{Theorem}[section]
\newtheorem{prop}[thm]{Proposition}
\newtheorem{cor}[thm]{Corollary}
\newtheorem{lem}[thm]{Lemma}
\newtheorem{rem}[thm]{Remark}
\newcommand{\T}{\mathbb{T}}

\newcommand{\D}{\mathbb{D}}

\newcommand{\cD}{{\mathcal{D}}}
\newcommand{\cS}{{\mathcal{S}}}

\newcommand{\cH}{{\mathcal{H}}}
\newcommand{\bT}{\mathbf{T_{\mu}}}
\newcommand{\bB}{\widetilde{\mu}}

\newcommand{\cA}{{\mathcal{A}}}

\newcommand{\hH}{\mathrm{H^{2}   }}

\newcommand{\caap}{\mathrm{cap}}

\newcommand{\Hol}{\mathrm{Hol}}

\newcommand\TT{\mathbb{T}}
\newcommand\DD{\mathbb{D}}

\newcommand\RR{\mathbb{R}}

\begin{document}
\author[Benazzouz, EL-Fallah, Kellay, Mahzouli]{H. Benazzouz, O. El-Fallah, K. Kellay, H. Mahzouli}

\address{\footnotesize Laboratoire Analyse et Applications URAC/03\\ Universit\'e Mohamed V Agdal-Rabat- \\  B.P. 1014 Rabat\\Morocco}
\email{elfallah@fsr.ac.ma, benazouz@fsr.ac.ma, mahzouli@fsr.ac.ma}

\address{\footnotesize IMB\\Universite Bordeaux I\\
351 cours de la Lib\'eration\\33405 Talence \\France}
\email{kkellay@math.u-bordeaux1.fr}
\thanks{\footnotesize First, second and fourth authors supported by CNRST (URAC/03).
Second and fourth authors supported by Acad\'emie Hassan II des sciences et techniques. Second and third authors
supported by PICS-CNRS}

\title{Contact points  and Schatten composition operators}

\maketitle

\begin{abstract} We study composition  operators on Hardy and Dirichlet spaces  belonging to  Schatten classes. We give some new examples and  analyse the size of contact set of the symbol of such operators.
\end{abstract}

\section{Introduction}
In this  paper, we consider composition operators  acting on  Hardy and weighted  Dirichlet spaces. 
Let $\DD$ be the unit disc.  Let $dA(z)=dxdy/\pi$ denote the normalized area measure on $\DD$. For $\alpha>-1$, $dA_\alpha$ will denote the finite measure on the unit disc $\DD$
given by
$$
dA_\alpha(z) := (1+\alpha)(1-|z|^2)^\alpha dA(z).
$$
The weighted Dirichlet space $\cD_\alpha$ ($0\leq \alpha\leq 1$) consists of those analytic functions on $\DD$ such that
$$ \cD_\alpha(f):=\int_\DD|f'(z)|^2 dA_\alpha(z)\asymp\sum_{n\geq 0}|\widehat{f}(n)|^2(1+n)^{1-\alpha}<\infty.$$
Note that the classical  Dirichlet space ${\cD}$ corresponds to $\alpha=0$ and ${\cD}_1= \hH$ is  the Hilbertian Hardy space.
Every function $f\in \cD_\alpha$ has non--tangential limits almost everywhere on the unit circle $\TT=\partial \DD$. If the non--tangential limit of $f$ at $\zeta\in \TT$ exists it also will be denoted by $f(\zeta)$.

Let $\varphi$ be a holomorphic self-map of $\D$.  The composition operator $C_\varphi$
on $\cD_\alpha$ is defined by
$$C_\varphi(f)=f\circ\varphi,\qquad f\in \cD_\alpha.$$
 \noindent For $s\in (0,1)$, the level set of  $\varphi$ is given by
$$E_\varphi(s)=\{\zeta\in \TT \text{ : }  |\varphi(\zeta)|\geq s\},$$
The contact set of $\varphi$ is $E_\varphi:=E_\varphi(1)$.\\
Let $\cH$ be a Hillbert space, a compact operator is said to belong in the Schatten class $\cS_p(\cH)$ if its sequence of singular numbers is in
the sequence space $\ell^p$.\\
  In section 2, we give a simple proof of Luecking's Theorem \cite{L,Z} about the  characterization for $p$--Schatten class of Toeplitz operators for $p\geq 1$. In section 3, we give a simple sufficient condition, in terms of the level set, which ensures the membership to $ {\cS}_p(\hH)$. This approach allows us to give explicit examples of composition operator belonging to Schatten classes. If the symbol is outer and the contact set is reduced to one point, we give an explicit complete characterization  to the membership to $ {\cS}_p(\hH)$.
  In the last section, we study the size of contact set of $\varphi$, when $C_{\varphi} \in {\cS}_p({\cD}_\alpha)$.
For a treatment of some questions addressed in this paper see also \cite{EKSY, EIN, GG, GG1,LLQP1,LLQP2,LLQP3, QS}.
\\

The notation $A\lesssim B$ means that there is a constant $C$ independent of the
relevant variables such that $A \leq CB$. We write   $A\asymp B$  if both $A\lesssim B$ and $B\lesssim A$.

\section{Luecking characterization for Schatten class of Toeplitz operators}

\subsection{Toeplitz operators on the Bergman spaces}
Let $\alpha>-1$. We denote by  $\cA_\alpha$ the Bergman space consisting of analytic functions $f$ on $\DD$  such that
$$ \|f\|_{\cA_\alpha}^2=\int_\DD|f(z)|^2dA_\alpha(z)<\infty.$$
The reproducing kernel of $\cA_\alpha$ is given by
$$K_w(z)=\frac{1}{(1-\bar{w}z)^{2+\alpha}},\qquad z,w\in \DD.$$
So we have
\begin{equation}\label{noyau}
f(w)=\langle f, K_w\rangle_{\cA_\alpha},\qquad f\in \cA_\alpha.
\end{equation}
In particular,
$$
\|K_w\|_{\cA_\alpha}^2= K_w(w)= \frac{1}{(1-|w|^2)^{2+\alpha}}.
$$
For a positive measure $\mu$ on the unit disc we associate the operator $\bT$ defined on the Bergman space $\cA_\alpha$ by
$${\bT} (f)(z):=\int_{\DD}K_w(z)f(w)d\mu(w)=\int_\DD\frac{f(w)}{(1-\bar{w}z)^{2+\alpha}} d\mu(w), \qquad f\in \cA_\alpha.$$
Let us denote $ k_w=K_w/\|K_w\|$ the normalized reproducing kernel at $w$. The Berezin transform of $\bT$  is given  by
$${\bB}(z):=\langle \bT k_z, k_z\rangle_{\cA_\alpha}=\int_\DD\frac{(1-|z|^2)^{2+\alpha}}{|1-\overline{w}z|^{4+2\alpha}} d\mu(w).$$
The hyperbolic measure in $\DD$ is given by
$$d\lambda(w)=(1-|w|^2)^{-2}dA(w).$$
It satisfies
\begin{equation}
\label{lambdaeq}
\displaystyle \int _{\DD} |\langle k _w, f\rangle_{\cA_\alpha}|^2 {{d}
\lambda(w)}= \frac{1}{1+\alpha}\|f\|_\alpha ^2,\qquad f\in \cA_\alpha.
\end{equation}
The dyadic decomposition of $\DD$ is the family $(R_{n,j})$ given by
$$
R_{n,j}=\Big\{ re^{i\theta} \in \DD: \ r_n\leq r<r_{n+1}, \frac{2\pi j}{2^n}\leq \theta < \frac{2\pi(j+1)}{2^n}\Big\}$$
 where $ j=0,1,..,2^{n}-1$ and
where $1-r_n=2^{-n}$.\\
Using the dyadic decomposition, it is not difficult to prove that $\bB\in L^p(\DD,d\lambda)$ if and only if
$$\displaystyle \sum_{n\geq 0}2^{(2+\alpha)np}\sum_{j=0}^{2^{n}-1}\mu(R_{n,j})^p < \infty.$$
For details see \cite {WX}.\\
The following result is due to Luecking \cite{L} ( an alternative proof is given by Zhu \cite{Z}). Here we give a simple proof of this result.
\begin{thm}\label{lemtopelitz}Let $p\geq 1$.  The following assertions are equivalent.
\begin{enumerate}
\item[(i)]$\bT \in \cS_p(\cA_\alpha)$,
\item[(ii)]$\bB\in L^p(\DD,d\lambda)$.
\end{enumerate}
\end{thm}
For the proof,  we need the following key inequality (see \cite[Lemma 2.1]{EKSY}).
\begin{lem}\label{besovestimate}Let $f\in \cA_\alpha^2$.
Then there exists a constant $C$, depending only on $\alpha$,  such that
\begin{equation}\label{lemmeEKSY}
|f(z)|^2\leq  C_\alpha\int_\DD \frac{(1-|w|^2)^{2+\alpha}}{|1-\overline{w}z|^{4+2\alpha}}|f(w)|^2 \, dA_{\alpha}(w), \quad z\in \DD.
\end{equation}
\end{lem}
\begin{proof} Let $f_n$ be an orthonormal basis of $\cA_\alpha$. We have by \eqref{noyau} and \eqref{lemmeEKSY},
\begin{align*}\sum_n \langle \bT f_n, f_n\rangle_{\cA_\alpha}^{p}= &{\sum_{n\geq 1}\Big( \int_\DD|f_n(z)|^2 d\mu(z)} \Big)^p\nonumber\\
\leq &C_\alpha^p \sum_{n\geq 1}\Big( \int_\DD \Big[  \int_\DD \frac{(1-|w|^2)^{2+\alpha}}{|1-\overline{w}z|^{4+2\alpha}}|f_n(w)|^2dA_\alpha(w) \Big]{ d\mu(z) } \Big)^p \nonumber\\
=& C_\alpha^p \sum_{n\geq 1}\Big( \int_\DD\bB(w)|f_n(w)|^2 dA_\alpha(w)\Big)
^p\nonumber\\
=&C_\alpha^p  \sum_{n\geq 1}\Big( \int_\DD\bB(w)d\nu_n(w)\Big)^p,
\label{leuk1}
\end{align*}
where $d\nu_n(w)=|f_n(w)|^2dA_\alpha(w)$. Note that $\nu_n(\DD)=\|f_n\|^2_{\cA_\alpha}=1$. By \eqref{lambdaeq} and Jensen's inequality, we have
\begin{eqnarray*}
\sum_n \langle \bT f_n, f_n\rangle_{\cA_\alpha}^{p}&\leq& C_\alpha^p
\sum_{n\geq 1} \int_\DD(\bB(w))^p |f_n(w)|^2dA_\alpha(w) \\
&=&C_\alpha^p \int_\DD(\bB(w))^p\sum_{n\geq 1}| f_n(w)|^2dA_\alpha(w)\\
&=&C_\alpha^p\; \|\bB\|_{L^p(\DD,d\lambda)}^p.
\end{eqnarray*}
The last equality comes from the fact that
$$
\displaystyle \sum_{n\geq 1}| f_n(w)|^2= \displaystyle \sum_{n\geq 1}\langle  K_w, f_n\rangle_{\cA_\alpha}^2 = \|K_w\|_{\cA_\alpha}^2= \frac{1}{(1-|w|^2)^{2+\alpha}}
$$
Conversely, since  $\bT\in \cS_{p}(\cA_\alpha)$, let  $ (s_n)_{n\geq 0}$ be the singular values of $\bT$ and $(e_n)_{n\geq 0}$ be the orthonormal sequence of the eigenfunctions of $\bT$ associated to $(s_n)_{n\geq 0}$. Using the spectral decomposition of $\bT$ ($\bT= \displaystyle \sum _{n\geq 1}s_n\langle \cdot ,e_n\rangle e_n$) and Jensen's inequality we obtain

\begin{eqnarray*}
\|\widetilde{\mu}\|_{L^p(\DD,d\lambda)}^{p}&=&
\int_\DD|\langle \bT  k_w, k_w\rangle_{\cA_\alpha}|^pd\lambda(w)\\
&=& \int_\DD\Big(\sum_{n\geq 0}s_n|\langle  k_w,e_n\rangle_{\cA_\alpha}|^2\Big)^p d\lambda(w) \\
&\leq & \int_\DD\sum_{n\geq 0}s_n^p|\langle  k_w,e_n\rangle_{\cA_\alpha}|^2 d\lambda(w)\\
&=& \frac{1}{1+\alpha} \sum_{n\geq 0}s_n^p.
\end{eqnarray*}
This  completes the proof.
\end{proof}

\subsection{Composition operators}
 Let $\alpha\geq 0$ and $\varphi  \in \Hol(\DD)$ such that $\varphi  (\DD)\subset \DD$.  The generalized counting Nevanlinna function of $\varphi$ is defined by
$$N_{\varphi  ,\alpha}(z)=\sum_{w\in \DD\text{ : } \varphi  (w)=z}(1-|w|)^\alpha, \qquad (z\in\DD).$$
Note that $N_{\varphi,0}(z)=n_\varphi(z)$ is the multiplicity of $\varphi$ at $z$  and $N_{\varphi,1}.=N_\varphi$  is equivalent to the usual  Nevanlinna counting function associated to $\varphi$.  For a Borel subset $\Omega$ of $\DD$,
we put
$$\mu_{\varphi,\alpha}(\Omega)=\int_{{\Omega}} N_{\varphi,\alpha}(z) dA(z).$$

It is well known that $C_\varphi\in \cS_p(\cD_\alpha)$ if and only if $C_\varphi^*C_\varphi \in \cS_{p/2}(\cD_\alpha)$. By a routine calculation (see \cite{WX}), there exists a rank one operator $R$ on $\cA _{\alpha}$ such that $C_\varphi^*C_\varphi$ and $\mathbf{T} _{\mu_{\varphi,\alpha}}+ R$ are similar. It implies that $C_\varphi\in \cS_p(\cD_\alpha)$ if and only if $\mathbf{T} _{\mu_{\varphi,\alpha}} \in \cS_{p/2}(\cA_\alpha)$, and the following result can be deduced easily from Theorem \ref {lemtopelitz} (see also \cite{WX}). \\
\begin{cor}\label{leucking} Let $p\geq 2$, $0\leq \alpha\leq 1$ and $\varphi$ holomorphic self-map on $\DD$. The following assertions are equivalent

\noindent (i) $C_\varphi\in \cS_p(\cD_\alpha)$,\\

\noindent (ii) $\displaystyle \sum_n 2^{(2+\alpha)np/2}\sum_{j=0}^{2^{n}-1} \Big[\mu_{\varphi,\alpha}(R_{n,j})\Big]^{p/2}<\infty$,\\

\noindent (iii) ${\widetilde{\mu }_{\varphi, \alpha}}\in L^{p/2} d\lambda(w).$

\end{cor}

\begin{rem}\label{llqp}
\noindent {\bf 1.}  Let $g$  be a positive  measurable function on $\DD$ and  a holomorphic self-map  $\varphi$ on $\DD$. By the change of variables formula \cite{S} we have
$$
\int_\DD (g\circ\varphi  ) (z)|\varphi  '(z)|^2dA_\alpha(z)=(1+\alpha)\int_\DD g(z) N_{\varphi  ,\alpha}(z)dA(z).
$$
Hence  the condition $(iii)$ becomes
\begin{equation}\label{changevar}
 \mathcal{I}_{\alpha,p}(\varphi)=(1+\alpha)^{p/2}\int_{\DD}\Big(\int_\DDÊ\frac{(1-|w|^2)^{2+\alpha}}{|1-\overline{w}z|^{4+2\alpha}}N_{\varphi,\alpha}(z)dA(z)\Big)^{p/2}{d\lambda(w)} <\infty.
\end{equation}
\noindent {\bf  2}. The pull back measure associated to $\varphi$ is defined by
$$m_\varphi(B):=|\{\zeta\in \TT\text{ : } \varphi(\zeta)\in B \text{ a.e}\}|, $$
here $B$ is a   Borel subset for $\overline{\D}$ and $|E|$
denotes the normalized Lebesgue measure of a Borelian subset $E$ of  $\TT$.

In  \cite{LLQP2}, Lef\`evre, Li, Queff\'elec and Rodriguez-Piazza showed that the classical  Nevanlinna counting function and the pull buck measure are connected as follows :  There exists two universal constants $C_1,C_2$ such that
 $$m_\varphi(W(\zeta, C_1h)\lesssim \sup_{z\in W(\zeta,h)\cap \DD}N_\varphi(z)\lesssim m_\varphi(W(\zeta,C_2 h)),\qquad \zeta\in \TT,\; h\in (0,1) $$
 where  $W(\zeta,h)=\{z\in \DD  \text{ : } 1-h\leq |z|<1 \text{ and } |\arg(z\bar\zeta)|\leq h\}$ are the Carleson boxes.

 Clearly, one can  formulate the membership to Schatten classes, in the case of the Hardy space, in terms of the pull back measure as follows,
$$C_\varphi\in \cS_p(\hH )\iff
\displaystyle \sum_n 2^{np/2}\sum_{j=0}^{2^{n}-1} \Big[m_{\varphi}(R_{n,j})\Big]^{p/2}<\infty.$$
Let  $W_{n,j}$ the dyadic Carleson box given by
$$W_{n,j}=\big\{z=re^{i\theta}\in \D  \text{ : }  1-2^{-n}\leq|z|<1\text{ and }  \frac{2\pi j}{2^n}\leq \theta < \frac{2\pi(j+1)}{2^n}\big\}, $$
where  $j=0,1,..,2^{n}-1.$
It was remarked in \cite{LLQP3} that
$$\displaystyle \sum_n 2^{np/2}\sum_{j=0}^{2^{n}-1} \Big[m_{\varphi}(R_{n,j})\Big]^{p/2}<\infty \iff
\displaystyle \sum_n 2^{np/2}\sum_{j=0}^{2^{n}-1} \Big[m_{\varphi}(W_{n,j})\Big]^{p/2}<\infty.$$
In this paper we will also use the earlier characterization of compactness du to B. MacCluer and J. Shapiro. They showed in  \cite{MCS} that $C_\varphi$  is compact on $\hH $ if and only if
$$ \sup_{\zeta\in \TT} m_\varphi(W(\zeta, h)) =o(h)(h\to 0).$$
\end{rem}
\section{Membership to $\cS_p(\hH )$}
\subsection{Membership to $\cS_p(\hH )$ through level sets}
 Note that   $C_\varphi  $ is in the Hilbert-Schmidt class in $\hH $ (i.e.   $C_\varphi  \in {\cS}_2(\hH ))$  if and only if
$$\sum_{n \geq 0}\|\varphi  ^n\|^{2}_{2 }=\frac{1}{2\pi}\int_\TT\frac{|d\zeta|}{1-|\varphi  (\zeta)|^2}\asymp \int_{0}^{1}\frac{|E_\varphi  (s)|}{(1-s)^2}\, ds<\infty.$$
Then the membership of composition operators to ${\cS}_2(\hH )$ is completely described by the level sets of their symbols. For $p>2$, it is proved in \cite {LLQP1} that there exists two symbols $\varphi, \psi$ such that
$|E_{\varphi}(r)|= |E_{\psi}(r)|$  for $r\in (0,1]$, $C_{\varphi}\in {\cS}_p({\hH})$ and $C_{\psi}\notin {\cS}_p({\hH})$. In the following proposition we give a sufficient condition in terms of the level sets which ensures the membership to Schatten classes. This allows us to give new examples of operators in ${\cS}_p({\hH})\setminus {\cS}_2({\hH})$ for $p>2$.
\begin{prop}\label{shat} Let $p\geq 2$. 
If
$$\int_0^1\frac{|E_\varphi(s)|^{p/2}}{(1-s)^{1+p/2}}ds<\infty, $$
then  $C_\varphi\in S_p(\hH )$.
\end{prop}

\begin{proof}  Since
$|E_\varphi(1-2^{-n})|= \sum_{j=0}^{2^n-1}m_\varphi(W_{n,j}), $
we get
$$ \sum_{j=0}^{2^n-1}m_\varphi(W_{n,j})^{p/2}\leq |E_\varphi(1-2^{-n})|^{p/2}.
$$
Hence,
\begin{eqnarray*}
\sum_n2^{np/2 }\sum_{j=0}^{2^n-1}m_\varphi(W_{n,j})^{p/2}
&\leq& \sum_{n}|E_\varphi(1-2^{-n})|^{p/2}\int_{1-2^{-n}}^{1-2^{-n-1}} \frac{ds}{(1-s)^{1+p/2}}\\
&\leq& \int_0^1\frac{|E_\varphi(s)|^{p/2}}{(1-s)^{1+p/2}}ds<\infty.
\end{eqnarray*}
By Remarks \ref{llqp}.2, we obtain $C_\varphi\in S_p(\hH )$.
\end{proof}
\begin{rem}\label{rem-}
If $C_\varphi\in S_p(\hH )$, then
$$\int_{0}^{1}\frac{|E_\varphi(r)|^{p/2}}{(1-r)^2}<\infty.$$
Indeed, write   $|E_\varphi(1-2^{-n})|= 2^{{n}}(\sum_{j=0}^{2^n-1}2^{-n}m_\varphi(W_{n,j})).$
So by Jensen's inequality, $$|E_\varphi(1-2^{-n})|^{p/2}\leq  2^{np/2 -n}\sum_{j=0}^{2^n-1}m_\varphi(W_{n,j})^{p/2}.$$
Since $C_\varphi\in S_p(\hH )$, by Remark \ref{llqp}.2, we have
$$
\int_0^{1}\frac{|E_\varphi(r)|^{p/2}}{(1-r)^2}\leq  \sum_n 2^n |E_\varphi(1-2^{-n})|^{p/2} \leq \sum_n2^{np/2 }\sum_{j=1}^{2^n}m_\varphi(W_{n,j})^{p/2}<\infty.
$$
\end{rem}

Now we are able to give some concrete examples.
Let $K$ be a closed subset of the unit circle $\T$.
Fix a non-negative function $h\in C^{1}[0,\pi]$  such that $h(0)=0$. We consider the outer function defined by
 \[
f_{h,K}(z) = \exp \bigg(- \int_{\TT} \frac{\zeta+z}{\zeta-z} \, h(d(\zeta, K)) \,\, |d\zeta| \bigg),
\]
where $d$ denotes the arc-length distance.
It is known that the non tangential limit of $f_{h,K}$ satisfies
\begin{equation}\label{exterieure}
|f_{h,K}(\zeta)|=e^{-h(d(\zeta,K))},\qquad \text{ a.e. on }\T.
\end{equation}
Given  $K\subset \TT$ and $t>0$, let us write $K_t=\{\zeta\text{ : } d(\zeta,K)\leq t\}$ and $|K_t|$
 denotes the Lebesgue measure of $K_t$.
 \begin{cor} \label {example} Let $p\geq 2$ and let $\varphi=f_{h,K}$.
\begin{enumerate}
\item If $$\int_0\frac{h'(t)}{h(t)^{1+p/2}}|K_t|^{p/2}dt<\infty,$$
then $C_{\varphi}\in S_p(\hH )$.
\item If  $C_{\varphi}\in S_p(\hH )$ then
$$\int_0\frac{h'(t)}{h(t)^{2}}|K_t|^{p/2}dt<\infty.$$
\end{enumerate}
\end{cor}

\begin{proof}
Since
\begin{eqnarray*}
|E_{\varphi}(s)|&=&|\{\zeta\in \TT\text{ : } e^{-h(d(\zeta,K))}\geq s\}|\\
&\asymp&|\{\zeta\in \TT\text{ : } d(\zeta,K)\leq h^{-1}(1-s)\}|\\
&=&|K_{h^{-1}(1-s)}|
\end{eqnarray*}
Proposition \ref{shat}  and Remark \ref{rem-} give the result.
\end{proof}

\noindent Note that there are several examples of composition operators which belong in ${\cS}_p(\hH)\setminus {\cS}_2(\hH)$ for $p>2$ (see \cite {CC, EIN, JON,LLQP1}). In all these examples the contact sets of their symbols is reduced to one point. Here we will construct examples with a large set of  contact points. To state our example we have to recall the definition of Hausdorff dimension.\\
Let $E$ be a closed subset of $\TT$. The Hausdorff dimension of $E$ is defined by
$$d(E)=\inf\{\alpha \text{ : } \Lambda_\alpha(E)=0\}$$
where $\Lambda_\alpha(E)$ is the $\alpha$--Hausdorff measure of $E$ given by
$$\Lambda_\alpha(E)=\lim_{\epsilon\to 0} \inf\Big\{\sum_{i}|\Delta_i|^\alpha\text{ : } E\subset\bigcup_i \Delta_i \text{ , } |\Delta_i|<\epsilon\Big\}.$$
\begin{cor} Let $ p> 2$ there exists an analytic self--map $\varphi$ of $\DD$ such that $\varphi\in A(\DD)$, the disc algebra, $C_\varphi\in \cS_p(\hH)\backslash \cS_2(\hH)$ and $d(E_\varphi)=1$.
\end{cor}
\begin{proof}
 It suffices to apply corollary  \ref {example} with $\varphi = f_{h,K}$ where
 $$h(t)= \frac{1}{\log ^2(e/t)}\ \ \mbox{ and}\ \  |K_t| \asymp \frac{1}{\log ^2(e/t)\log \log (e^2/t)}.$$\qedhere
\end{proof}
Other type of examples are constructed by Gallardo-Gonzalez \cite{GG}. They proved that there exists a univalent symbol $\varphi$ such that $C_\varphi$ is compact on $\hH $ and such that the Hausdorff dimension of $E_{\varphi}$ is  equal to one. This result can not be extended to Schatten classes. Indeed, we have the following result

\begin{prop} Let $p\geq2$. If $\varphi$ is univalent function such that $C_\varphi\in \cS_2(\hH)$ then $$d(E_\varphi)\leq \frac{p}{p+2}.$$
\end{prop}
For the proof see Remark \ref{exempleunivalent}.

\subsection{Contact set reduced to one point} In this subsection we will consider outer functions $\varphi$ which their contact set is reduced to one point. In this case, and under some regularity conditions, we give a concrete necessary and sufficient condition for the membership to Schatten classes.

Let $h$ be a continuous increasing function defined on $[0,\pi]$ such that $h(0) =0$. We extend it to an even $2\pi$-periodic function on $\RR$.
We say that the function $h$ is admissible if $h$ is differentiable, $h(2t)\asymp h(t)\asymp th'(t)$, and $h$ is concave or convex.

Let $\varphi$ be the outer function on $\DD$ such that
$$|\varphi (e^{it})|=e^{-h(t)},\quad \text{ a.e on } (0,\pi).$$

 We have the following:

\begin{thm}\label {Onepointtheorem} Let $h$ be an admissible function such that  $t^2=o(h(t))$ and
$$\displaystyle h(\theta)=o\big(\theta\int_\theta^\pi\frac{h(t)}{t^2}\big)(t\to 0).$$
Then
\begin{enumerate}
\item  $C_{\varphi}$ is compact if and only if
$$\int_0^\pi\frac{h(t)}{t^2}dt=\infty.$$
\item Let $p>0$, then
 $C_{\varphi} \in S_p (\hH)$  if and only if
$$\displaystyle \int _0^\pi\frac{dt}{h(t)\left(\displaystyle \int _t^\pi\frac{h(s)}{s^2}ds\right) ^{p/2 -1} }< +\infty.$$\\
\end{enumerate}
\end{thm}
As an immediate consequence of this theorem we obtain

\begin{cor} 1. There exists compact composition operator $C_\varphi$ on $\hH $ but belongs to none Schatten class.

2. Let $q>0$ there exists a compact composition operator $C_\varphi$ such that
$$C_\varphi\in \bigcap_{p> q} \cS_p(\hH )\backslash \cS_q(\hH ).$$
\end{cor}
To prove our theorem,
we need some notions. Let $\widetilde{h}$ be  the harmonic conjugate of $h$. It is defined by
$$
\widetilde{h}(\theta)=\lim_{\varepsilon\to 0}\frac{1}{2\pi}\displaystyle \displaystyle \int _{\varepsilon}^{\pi}\frac{h(\theta +t)-h(\theta-t)}{\tan(t/2)}dt.
$$
 The Hilbert transform of $h$ will be denoted by $Hh$ and is given  by
$$
Hh(\theta)=\lim_{\varepsilon\to 0}\displaystyle\frac{1}{\pi}  \int _{\varepsilon}^{\pi}\frac{h(\theta +t)-h(\theta-t)}{t}dt
$$
We will also need the following auxiliary function
$$\Psi(\theta):=  \frac{1}{\pi}\int_{2\theta}^{\pi-2\theta}h'(s)\log \frac{s+\theta}{s-\theta}ds.$$
The following estimates of $\widetilde{h}$ is the key of the proof of our theorem.
\begin{lem} \label {Hilbertestimate} Let $h$ be an admissible function. 
 There exists $a,b>0$ such that
$$
  \Psi(\theta)\leq \widetilde{h}(\theta)\leq\Psi(\theta)+ ah(\theta)+b\theta^2,\qquad \theta\in [0,\pi/4].
$$
\end{lem}
\begin{proof}
First let's estimate the Hilbert transform of $h$. Under our assumptions, the Hilbert transform can be written as follows
$$
Hh(\theta)=\frac{1}{\pi} \displaystyle \int _0^{\pi}\frac{h(\theta +t)-h(\theta-t)}{t}dt.
$$
We split the integral into three parts
\begin{align*}
Hh(\theta)& = \frac{1}{\pi} \displaystyle \int _0^{\theta}\frac{h(\theta +t)-h(\theta-t)}{t}dt
+ \frac{1}{\pi} \displaystyle \int _\theta^{\pi -\theta}\frac{h(\theta +t)-h(t-\theta)}{t}dt\\
& \; \; \; \; \; +
\frac{1}{\pi} \displaystyle \int _{\pi -\theta}^{\pi}\frac{h(2\pi-\theta -t)-h(t-\theta)}{t}dt\\
&= A+B+C.\\
\end{align*}
Since $h$ increases on $(0,\pi)$, it is clear that $A, B, C$ are positive.
First, we will prove that
\begin{equation}\label{eq:A+C}
A+C =O( h(2\theta)+\theta^2).
\end{equation}
 By concavity or convexity, we have
 $$h(t+\theta)-h(\theta-t)\leq 2t\max(h'(\theta-t),h'(\theta+t)).$$
Hence,
 $$A= \frac{1}{\pi} \displaystyle \int _0^{\theta}\frac{h(\theta +t)-h(\theta-t)}{t}dt \leq \frac{2}{\pi} h(2\theta)$$
 By a change of variables and convexity, we get
 \begin{eqnarray*}
 C&=&\frac{1}{\pi}\int_{0}^{\theta}\frac{h(\pi-\theta+u)-h(\pi-\theta-u)}{\pi-u}du\\
&=&\frac{1}{\pi}\int_{0}^{\theta}\frac{2u\max(h'(\pi-\theta-u),h'(\pi-\theta+u)) }{\pi-u}du \\
&\leq & \frac{8\theta^2}{3\pi^2}\sup_{\theta\in [\pi/2,\pi]}|h'(t)| .
\end{eqnarray*}
Hence \eqref{eq:A+C} is proved. Now we have to estimate $B$.
  We have
  \begin{align*}
  \pi B&=\int_{0}^{\pi}\frac{\chi_{[\theta,\pi-\theta]}}{t}\Big(\int_{t-\theta}^{t+\theta}h'(s)ds\Big)dt\\
  &=\int_{0}^{\pi}h'(s)\int_{s-\theta}^{s+\theta}\frac{\chi_{[\theta,\pi-\theta]}}{t}dt\\
  &=\int_{\theta}^{2\theta}h'(s)\log\frac{s+\theta}{\theta}ds+\int_{2\theta}^{\pi-2\theta}h'(s)\log \frac{s+\theta}{s-\theta}ds+\int_{\pi-2\theta}^{\pi}h'(s)\log\frac{\pi-\theta}{s-\theta}ds\\
  &=B_1+B_2+B_3
  \end{align*}
Note that
$$
B_1+B_3 =O( h(2\theta)+\theta^2).
$$
Indeed, we have
$$B_1\leq \log 3 (h(2\theta)-h(\theta))$$
and
$$B_3\leq 2\theta\log\frac{\pi-\theta}{\pi-3\theta}\sup_{[\pi/2,\pi]}|h'(s)|\leq\frac{4^2\theta^2}{\pi}\sup_{[\pi/2,\pi]}|h'(s)|.$$
Hence the estimate of the Hilbert transform  follows from $B_2$ and we have
$$ \Psi(\theta)\leq Hh(\theta)\leq\Psi(\theta)+ c_1h(2\theta)+c_2\theta^2,\qquad \theta\in [0,\pi/4]. $$
Since
$$\frac{1}{\tan(t/2)}-\frac{1}{t}=-\frac{t}{3}+o(t^2),$$
as before
$$|Hh(\theta) - \widetilde{h}(\theta)| = O(h(2\theta)+\theta^2) \qquad \theta\to 0.$$
The proof is complete.
\end{proof}

\begin{rem} \label{remarque}
1.  If  $\displaystyle \int_0^\pi\frac{h(t)}{t^2}dt=\infty$, then
$$\theta^2+h(\theta)=O\Big(\theta\int_\theta^\pi\frac{h(t)}{t^2}dt\Big),\quad \theta\to 0+.$$

2. Note that if $h$ is an admissible function and  $\displaystyle \int_0^\pi\frac{h(t)}{t^2}dt=\infty$ then
$$\Psi(\theta)= \int_{2\theta}^{\pi-2\theta}h'(s)\log \frac{s+\theta}{s-\theta}ds\asymp\theta\int_\theta^\pi\frac{h(t)}{t^2}dt.$$
Observe that, the function $\Psi$ is increasing, $\Psi(0)=0$, and satisfies the following properties
 \begin{itemize}
   \item $\Psi(t)\asymp\Psi(2t)$
   \item $\Psi'(t)\asymp\displaystyle \int_\theta^\pi\frac{h(t)}{t^2}dt$
   \item $(\Psi^{-1})'(2t)\asymp (\Psi^{-1})'(t)$.
 \end{itemize}
\end{rem}
\subsection*{Proof of Theorem }
 1) Let $m_\varphi$ be the pull back measure associated to the function $\varphi$ and let $W(1,\delta)= \{z \in \DDÊ\text{ :  } 1-|z| < \delta,\ |\arg (z)|< \delta\}$ be a Carleson box.

     Note that
 \begin{eqnarray*}
 m_\varphi(W(1,\delta))&= &|\{\theta \in (-\pi,\pi) :\ |\varphi ^*(e^{i\theta})|\in W(1,\delta)\}\\
 & \asymp & |\{\theta \in (0,\pi):\ h(\theta)< \delta, \;\; \widetilde{h} (\theta )< \delta\}|\\
  & \asymp & |\{\theta \in (0,\pi):\Psi(\theta)< \delta\}|\\
 &\asymp& \Psi ^{-1}(\delta).
\end{eqnarray*}
Recall that $C_\varphi$ is compact if and only if $m_\varphi(W(1,\delta))= o(\delta)$.
It follows that $C_\varphi$ is compact if  $\Psi ^{-1}(\delta)=o(\delta)$ as $\delta\to 0$, which is equivalent to
$$\displaystyle \int _0\frac{h(t)}{t^2}dt = +\infty.$$

Conversely, suppose that $\displaystyle \int _0^\pi\frac{h(t)}{t^2}dt < +\infty$. It is clear that $h(t) = o(t)$. Note that $\theta=O(\widetilde{h}(\theta))$. Indeed by convexity
 \begin{align*}
\widetilde{h}(\theta) &\geq  \frac{1}{2\pi}  \int _{2\theta}^{\pi -\theta}\frac{h(\theta +t)-h(t-\theta)}{\tan (t/2)}dt\\
&\geq  2\theta \displaystyle \int _{2\theta}^{\pi -\theta}\frac{\max(h'(t-\theta),h'(t+\theta))}{\tan (t/2)}dt\asymp  \theta \displaystyle \int _{\theta}^{\pi}\frac{h(t)}{t^2}dt
\asymp  \theta
 \end{align*}
So in this case $m_\varphi(W(1,\delta)) \asymp \delta$ and $C_\varphi$ is not compact.
 \vspace{1em}

 2) To prove the second assertion, we will estimate $m_\varphi(W_{n,j})$, where
 $$W_{n,j}= W(e^{i2\pi j/2^n}, 1/2^n)=\{z\in \DD\text{ : } 1-|z|<1/2^n ,\ j/2^n\leq  \arg z< (j+1)/2^n\}.$$
Let
$$\Omega_{n,j}=\{\theta:\ h(\theta)<1/2^n,\ j/2^n \leq \widetilde{h}(\theta)<(j+1)/2^n\}.$$
 We have  $m_\varphi(W_{n,j})= |\Omega_{n,j}|.$
  By Lemma \ref {Onepointtheorem}, there exists $\kappa>0$ such that
$$\Psi(\theta)\leq  \widetilde{h}(\theta)\leq \Psi(\theta)+\kappa h(\theta).$$
Let $$A_{n,j}:=\{\theta :\ h(\theta)<1/2^n,\  j/2^n \leq \Psi(\theta)<(j+1)/2^n\}.$$
 Hence for $j\geq [\kappa]+1$,
 $$\Omega_{n,j}\subset \{\theta :\ h(\theta)<1/2^n,\  (j-\kappa)/2^n \leq \Psi(\theta)<(j+1)/2^n\}= \bigcup_{l=j-[\kappa]}^{j} A_{n,l}$$
 and for $j\leq [\kappa]$,
 $$\Omega_{n,j}\subset \{\theta :\ h(\theta)<1/2^n,\   \Psi(\theta)<(j+1)/2^n\}= \bigcup_{l=0}^{j} A_{n,l}.$$
 Note that for $\theta \in A_{n,j}$, we have $A_{n,j}=\empty$ for $j> J_n=2^n\Psi(h^{-1}(1/2^n))$. We obtain

 $$\sum_n 2^{np/2}\sum_{j=0}^{2^{n}-1} \Big[m_\varphi(W_{n,j})\Big]^{p/2}\lesssim  \sum_n 2^{np/2}\sum_{j=0}^{J_n} |A_{n,j}|^{p/2}.$$
 Recall that $(\Psi^{-1})'(2t)\asymp (\Psi^{-1})'(t)$. By Remark \ref{remarque}.2, we have

 \begin{multline*}
 \sum_n 2^{np/2}\sum_{j=0}^{J_n} |A_{n,j}|^{p/2}\asymp\sum_n 2^{np/2}\sum_{0}^{J_n}\Big(\int_{j/2^n}^{(j+1)/2^n}(\Psi^{-1})'(t)dt\Big)^{p/2}\\
 \asymp\sum_n \sum_{0}^{J_n}\Big[(\Psi^{-1})'(j/2^n)\Big]^{p/2}
\asymp\sum_n \int_{0}^{J_n}\Big[(\Psi^{-1})'(s)\Big]^{p/2}ds\\
\asymp\sum_n \sum_{0}^{J_n} 2^n \int_{j/2^n}^{(j+1)/2^n}\Big[(\Psi^{-1})'(t)\Big]^{p/2}dt
\asymp\sum_n  2^n \int_{0}^{J_n/2^n}\Big[(\Psi^{-1})'(t)\Big]^{p/2}\\
\asymp\sum_n  2^n \int_{0}^{\Psi^{-1}(h^{-1}(1/2^n))}\Big[(\Psi^{-1})'(t)\Big]^{p/2}
\asymp\sum_n  2^n \sum_{k=n}^{\infty}\int_{\Psi^{-1}(h^{-1}(1/2^k))}^{\Psi^{-1}(h^{-1}(1/2^{k+1}))}\Big[(\Psi^{-1})'(t)\Big]^{p/2}\\
\asymp\sum_{k=0}^{\infty}  2^k \int_{\Psi^{-1}(h^{-1}(1/2^k))}^{\Psi^{-1}(h^{-1}(1/2^{k+1}))}\Big[(\Psi^{-1})'(t)\Big]^{p/2}
\asymp\int_{0}^{1}\frac{\Big[(\Psi^{-1})'(t)\Big]^{p/2} }{h\circ \Psi^{-1}(t)}dt\\
\asymp\int_{0}^{1}\frac{1}{h(u)}\frac{1}{\Big[\Psi'(u)\Big]^{\frac{p}{2}-1}}du
\asymp\int_0^1\frac{1}{h(u)\Big(\displaystyle\int_{u}^{1}\displaystyle\frac{h(s)}{s^2}ds\Big)^{\frac{p}{2}-1}}du.
\end{multline*}

Conversely,  let $0<c_1<1$
 $$B_{n,j}:=\{\theta :\  h(\theta)<(1-c_1)/(\kappa2^{n})\text{ : } j/2^n \leq \Psi(\theta)<(j+c_1)/2^{n}\},$$
By Lemma \ref {Hilbertestimate} $B_{n,j}\subset \Omega_{n,j}$. The rest of the proof runs in the same way as before.
 \hfill{$\Box$}
\begin{rem}
Note that H. Queffelec and K. Seip studied in \cite {QS} the asymptotic behavior of the singular values of composition operators with symbol having one point as  contact set.
\end{rem}

\section{ Schatten class $\cS_p(\cD_\alpha)$ and level sets}

Let $\varphi$ be a holomorphic self map of $\DD$. In this section, we discuss the size of $E_\varphi$, when $C_\varphi \in {\cS}_p(\cD_\alpha)$.

Given a (Borel) probability measure $\mu$ on $\TT$, we define its $\alpha$-energy, $0\leq \alpha < 1$, by
$$I_\alpha(\mu)=\sum_{n=1}^\infty \frac{|\widehat{\mu}(n)|^2}{n^{1-\alpha}}.$$
For a closed set $E \subset \TT$, its $\alpha$-capacity $\caap_\alpha (E)$ is defined by
$$\caap_\alpha (E):= 1/\inf\{I_\alpha(\mu) \text{ : } \mu \text{ is a probability measure on } E \}.$$
For Borelian set $E$ of the unit circle its $\alpha$-capacity is defined as follows
$$
\caap_\alpha (E):=\sup \{\caap_\alpha (F):\ F\subset E, \ F\  \mbox{closed} \}.
$$
Note that if $\alpha=0$,  $\caap:= \caap_0 $ is equivalent to the classical logarithmic capacity.  There is a connection between the $\alpha$--capacity and the Hausdorff dimension. In fact the capacitary dimension of $E$ is the supremum of $\alpha> 0$ such that $\caap_\alpha(E)>0$. By Frostman's Lemma \cite{KS} the capacitary dimension is equal to the Hausdorff dimension for compact sets.
Let us mention the result obtained by Beurling  in \cite{B} (and extended by Salem Zygmund \cite{C,KS}), which reveals an important connection between $\alpha$-capacities and weighted Dirichlet spaces. These results can be stated as follows:
Let $f\in \cD_\alpha$, the radial limit of $f$ satisfies capacitary weak-type inequality
 $$\caap_\alpha\{\zeta\in \TT\text{ : } |f(\zeta)|\geq t\}\leq A\frac{ \|f\|_\alpha^2}{t^2}.$$
In particular  $$\caap_{\alpha}(\{\zeta\in \TT\text{ : } f(\zeta) \text{ does not  exist}\})=0.$$

Our main result in this section  is the following theorem.
 \begin{thm}\label{theorem1}
Let $\varphi$ be a holomorphic self--map of $\DD$, $\alpha \in (0,1)$ and $p\leq 2/(1- \alpha)$. If $C_\varphi\in \cS_p(\cD_\alpha)$ then $\caap_\alpha(E_\varphi)=0$.
\end{thm}
For the proof we need the following lemmas.
\begin{lem}\label{xlog}
If
\begin{equation}\label{inegalite1}
\int_\DD \frac{|\varphi'(z)|^2 dA_\alpha(z)}{(1-|\varphi(z)|^2)^2\log1/(1-|\varphi(z)|^2)}<\infty,
\end{equation}
then
$\caap_\alpha(E_\varphi)=0$
\end{lem}

\begin{proof}First, note that
\begin{equation}\label{xlog1}
\frac{1}{(1-x^2)^2\log e/(1-x^2)}\asymp \sum_{n\geq 0} \frac{1+n}{\log e(1+n) } x^{2n},\qquad x\in (0,1).
\end{equation}
Indeed,
$$
  \sum_{n\geq 0}\frac{1+n}{\log e(1+n) } x^{{2n}}\geq \sum_{\frac{1}{1-{x^2}}\leq n\leq \frac{2}{1-{x^2
  }}} \frac{1+n}{\log e(1+n) } x^{{2n}}
  \asymp\frac{1}{(1-x^2)^2\log e/(1-x^2)},
$$
  and
$$ \sum_{n\geq 0}\frac{1+n}{\log e(1+n) } x^{{2n}}\leq
 \sup_n  \frac{1+n}{\log e(1+n) }x^n\sum_m x^m
 \asymp\frac{1}{(1-x)\log e/(1-x)}\frac{1}{1-x}.
 $$
 By  \eqref{xlog1}, we have
\begin{eqnarray*}
\int_\DD \frac{|\varphi'(z)|^2 dA_\alpha(z)}{(1-|\varphi(z)|^2)^2\log1/(1-|\varphi(z)|^2)}
&\asymp&\sum_{n\geq 0} \frac{1+n}{\log (1+n)} \int_\DD  {|\varphi'(z)|^2}{|\varphi(z)|^{2n}}dA_\alpha(z)\\
&=&\sum_{n\geq1}
\frac{\cD_\alpha(\varphi^{n})}{(1+n)\log(1+n)}.
\end{eqnarray*}
So  \eqref{inegalite1} implies that  $\liminf_n\|\varphi^n\|_\alpha=0$.  On the other hand, the weak capacity inequality  gives
$$\caap_\alpha(E_\varphi)= \caap_\alpha(E_{\varphi^n}) \leq c_\alpha\| \varphi^n\|^2_\alpha.$$
Now let $n\to  \infty$,  we get our result.
\end{proof}
 \begin{lem}\label{Lem:intet}
  Let $d>0$, $c>-1$ and $\sigma\geq 0$,  then
\begin{equation}\label{integral}
 \int_\DD \frac{dA_c(w)}{|1-z\overline{w}|^{2+c+d} |\log(1-|w|^2)|^\sigma} \asymp\frac{1}{(1-|z|^2)^d|\log(1-|z|^2)|^\sigma}
 \end{equation}
\end{lem}
\begin{proof} Let $w=|w|e^{it}$, by \cite[Lemma 3.2]{Z}, we have
$$ \int_{0}^{2\pi} \frac{dt}{|1-z|w|e^{-it}|^{2+c+d}}\asymp \frac{1}{(1-|zw|)^{1+c+d}}.$$
Using this result the lemma follows from a direct computation.
 \end{proof}
 \subsection{Proof of Theorem} Let $\beta=2/(p-2)$ and
$$d\mu(w)=\frac{dA(w)}{(1-|w|)|\log(1-|w|)|^{1+\beta}}.$$
We have $\mu(\DD)<\infty$, so for   $p\geq 2$, by Jensen inequality   and by Lemma \ref{Lem:intet},  we obtain
\begin{multline*}
\Big( {\int_\DDÊ\frac{|\varphi'(z)|^2dA_\alpha(z)}{(|1-|\varphi(z)|^2)^{1+\alpha +2/p
}|\log(1-|\varphi(z)|^2)|} \Big)^{\frac{p}{2}}}\\
\asymp\Big(\int_\DD\int_{\DD}Ê\frac{(1-|w|^2)^{2+\alpha-\frac2p  } |\log(1-|w|)|^{(1+\beta)\frac2p} |\varphi'(z)|^2  }{|1-\overline{w}\varphi(z)|^{4+2\alpha}}dA_\alpha(z)\frac{d\mu(w)}{\mu(\DD)}\Big)^{\frac{p}{2}} \\
\lesssim
 \int_{\DD}\Big(\int_\DDÊ\frac{(1-|w|^2)^{2+\alpha-\frac2p}|\log(1-|w|)|^{(1+\beta)\frac2p} |\varphi'(z)|^2 }{|1-\overline{w}\varphi(z)|^{4+2\alpha}}dA_\alpha(z)\Big)^{\frac{p}{2}}\frac{d\mu(w)}{\mu(\DD)}\asymp {\mathcal{I}_{\alpha,p}(\varphi)}.
 \end{multline*}
 Since $\alpha+2/p+1\geq 2$, by   Remarks \ref{llqp}.2 and  Lemma \ref{xlog} we get the result.
\hfill $\Box$

As a consequence, we obtain the following corollary.
\begin{cor} Let $0<\alpha\leq 1$ and $\varphi$ be a holomorphic self--map of $\DD$. Suppose that $\varphi$ is univalent. If $C_\varphi\in \cS_p(\cD_\alpha)$ then $\caap_{\frac{p\alpha}{2+p\alpha}}(E_\varphi)=0$.
\end{cor}
\begin{proof} Since $\varphi$ is univalent, $N_{\varphi,\beta}=(N_{\varphi})^{\beta}$. By \cite{PP},
$C_\varphi\in S_{p}(\cD_{\beta})$ if and only if $N_{\varphi,\beta}\in L^{p/2}(\DD,d\lambda)$. Using these observations, it is clear that $C_\varphi\in \cS_p(\cD_\alpha)$ if and only if $C_\varphi\in \cS_{p/\gamma}(\cD_{\alpha\gamma})$. Let $\gamma = {2+p\alpha}/{ p}$, since $C_\varphi\in \cS_p(\cD_\alpha)$,  $C_\varphi\in \cS_{2+p\alpha}(\cD_{\frac{p\alpha}{2+p\alpha}})$. The result follows from Theorem \ref{theorem1}.
\end{proof}
\begin{rem}\label{exempleunivalent}If $\varphi$ is univalent function and  $C_\varphi\in \cS_p(\hH )$ (here $\alpha=1$) then $\caap_{\frac{p}{2+p}}(E_\varphi)=0$.
\end{rem}
\begin{prop} If $C_\varphi$ is bounded on $\cD_\alpha$ and  $\varphi(\DD)$ is contained in a polygon of the unit  disc,  then  $\caap_\alpha(E_\varphi)=0$.
\end{prop}
\begin{proof}
Since  $C_\varphi$ is bounded on $\cD_\alpha$,
$$\sup_{\xi\in \TT}\mu_{\varphi,\alpha}(W(\xi,h))=O(h^{2+\alpha})(h\to 0).$$
see \cite{EKSY, KL}. Hence $\mu_{\varphi,\alpha}(R_{n,j})=O(1/2^{(2+\alpha)n})$. Suppose that   $\varphi(\DD)$  is contained in a polygon of the unit  disc,
then for all $n$ we have  $\mu_{\varphi,\alpha}(R_{n,j})=0$ uniformly on $n$ except for a finite number of $j$. Then there exists $J$ such that for all $n$ and all $j\geq J$,  $\mu_{\varphi,\alpha}(R_{n,j})=0$ and so
$$\sum_{n\geq 1} 2^{2n}\sum_{j}\mu_{\varphi,\alpha}(R_{n,j})<\infty.$$
Now, we get
$$\int_\DD\frac{|\varphi'(z)|}{(1-|\varphi(z)|^2)^2}dA_\alpha(z)=\sum_n 2^{2n}\sum_j\Big(\int_{R_{n,j}}N_{\varphi,\alpha}(z)dA(z)\Big)<\infty,$$
and by Lemma \ref{xlog} we obtain
$\caap_\alpha(E_\varphi)=0$.
\end{proof}
Gallardo-Gonz\'alez \cite[Corollary 3.1]{GG1} showed that for all $\alpha \in (0,1]$ there exists a compact composition operator, $C_\varphi$, on $\cD_\alpha$
such that the Hausdorff dimension of $E_\varphi$ is one. In the case of the classical Dirichlet space ($\alpha =  0$) the situation is different as showed in the following proposition.
\begin{prop} If  $C_\varphi$ is compact on the Dirichlet space $\cD$, then $E_\varphi$ as vanishing Hausdorff dimension.
\end{prop}
\begin{proof}
Let  $K_\lambda(z)=\log 1/(1-z\overline{\lambda})$ be the reproducing kernel  of  $\cD$ and let $k_\lambda(z)=K_\lambda(z)/(\log 1/1-|\lambda|^2)^{1/2}$ the normalized reproducing kernel. If $C_\varphi$ is compact then $C^*_\varphi$ is as well  and  $C^{*}_{\varphi}(k_\lambda)\to 0,\quad |\lambda| \to 1$. Hence
$$\frac{k_{\varphi(\lambda)}(\varphi(\lambda))}{k_{\lambda}(\lambda)}=\frac{|\log (1-|\varphi(\lambda)|^2)|}{|\log (1-|\lambda|^2|)}\to 0,\qquad |\lambda|\to 1.$$
Then $(1-|\lambda|^2)^\beta/(1-|\varphi(\lambda)|^2)^2$ is bounded for all  ${\beta}\in(0,1]$. So
$$\int_\DD \frac{|\varphi'(z)|^2}{(1-|\varphi(z)|^2)^2}dA_\alpha(z)\leq C\int_\DD|\varphi'(z)|^2dA(z)=\cD(\varphi)<\infty.$$
By Lemma \ref{inegalite1}, $\caap_\alpha(E_\varphi)=0$ for all $\alpha$ and hence $d(E)=0$ (see \cite{KS}).
\end{proof}


\small
\begin{thebibliography}{99}


\bibitem{B}
A. Beurling, Ensembles exceptionnels, Acta. Math. 72 (1939), 1--13.

\bibitem{C}
Carleson, L.,  Selected Problems on Exceptional Sets. Van Nostrand, Princeton NJ, 1967.

\bibitem{CC} T. Carroll, C.C. Cowen, Compact composition operators not in the Schatten classes, J. Operator Theory 26 (1991)
109-120.

\bibitem{EKSY}  El-Fallah, O.;  Kellay, K.; Shabankhah, M. and H. Youssfi. Level sets and Composition operators on the Dirichlet space. J. Funct. Anal.,  260 (2011) 1721-1733

\bibitem{EIN} El-Fallah, O;  El Ibbaoui, M; Naqos, Composition operators with univalent symbol in Schatten classes. Preprint
\bibitem{GG}   Gallardo-Guti\'errez, Eva A. and Gonz\'alez, Maria J., Exceptional sets and Hilbert-Schmidt composition
operators. J. Funct. Anal. 199 (2003) 287-300.


\bibitem{GG1}  Gallardo-Guti\'errez, Eva A. and  Gonz\'alez, Maria J.,  Hausdorff measures, capacities and compact
composition operators. Math. Z., 253 (2006), 63--74.

\bibitem{G} Garnett, J., Bounded analytic functions, Academic Press, New York, 1981.

\bibitem{JON} M. Jones, Compact composition operators not in the Schatten classes, Proc. Amer. Math. Soc. 134 (2006) 1947-
1953.

\bibitem{KL} K. Kellay; P. Lef\`evre, Compact composition operators on weighted Hilbert spaces of analytic functions.
J. Math. Anal. Appl.,  386 (2) (2012)  718-727.


\bibitem{KS}J. P. Kahane, R. Salem. Ensembles parfaits et s\'eries trigonom\'etriques.  Hermann,  Paris, 1963.


\bibitem{LLQP} P. Lef\`evre; D. Li; H. Queff\'elec;  L. Rodriguez-Piazza, Approximation numbers of composition operators on the Dirichlet space.  arXiv:1212.4366


\bibitem{LLQP1} P. Lef\`evre; D. Li; H. Queff\'elec;  L. Rodriguez-Piazza, Compact composition operators on the Dirichlet space and capacity of sets of contact points . J. Funct. Analysis 624 (2013) no 4, 895--919.


\bibitem{LLQP2} P. Lef\`evre; D. Li; H. Queff\'elec;  L. Rodriguez-Piazza,
Nevanlinna counting function and Carleson function of analytic maps. Math Ann., 351, no2 (2011), 305-326.


\bibitem{LLQP3}P. Lef\`evre; D. Li; H. Queff\'elec;  L. Rodriguez-Piazza, Some examples of compact composition operators on ${\hH}$. J. Funct. Anal., Volume 255, Issue 11(2008), 3098-3124 .


\bibitem{LQP} D. Li; H. Queff\'elec;  L. Rodriguez-Piazza, Estimates for approximation numbers of
some classes of composition operators on the Hardy space. Ann. Acad. Sci. Fenn. Math, to appear.


\bibitem{QS}  H. Queff\'elec; K. Seip, Decay rates for approximation numbers of composition operators. arXiv: 1302.4116, 2013.


\bibitem{L}  Luecking, D.,  Trace ideal criteria for Toeplitz operators, J. Funct. Anal. 73 (1987) 345Ð368.


\bibitem{MCS} B. MacCluer and J.H. Shapiro, Angular derivatives and compact composition operators on the Hardy
and Bergman spaces, Canadian J. Math. 38 (1986), 878--906.


\bibitem{PP} J. Pau, P. A. Perez ,Composition operators acting on weighted Dirichlet spaces, J. Math. Anal. Appl.
401 (2013), no. 2, 682--694.


\bibitem{S} J. H. Shapiro, Composition operators and classical function theory, Springer Verlag, New York 1993.


\bibitem{WX}Wirths, K.-J. and  Xiao,  J., Global integral criteria for composition operators.
J. Math. Anal. Appl., 269 (2002), 702--715.

\bibitem{Z} Zhu, K., Operator theory in function spaces. Monographs and textbooks in pure and applied mathematics, 139, Marcel Dekker, Inc (1990).
\end{thebibliography}
\end{document}